\documentclass[a4paper,11pt]{article}
\usepackage{amsmath,amssymb,amsfonts,latexsym, amsthm}

\newcommand{\R}{\mathbb {R}}

\newcommand{\I}{\infty}

\newcommand{\grad}{\nabla}
\newcommand{\bn}{\mathbb{B}^{N}}
\newcommand{\rn}{\mathbb{R}^{N}} 
\newcommand{\hn}{\mathbb{H}^{N}}

\newtheorem{theorem}{Theorem}[section]
\newtheorem{lemma}[theorem]{Lemma}

\newtheorem{proposition}[theorem]{Proposition}
\newtheorem{remark}[theorem]{Remark}
\newtheorem{corollary}[theorem]{Corollary}
\numberwithin{equation}{section}

\newcommand{\be} {\begin{equation}}
\newcommand{\ee} {\end{equation}}
\newcommand{\bea} {\begin{eqnarray}}
\newcommand{\eea} {\end{eqnarray}}
\newcommand{\Bea} {\begin{eqnarray*}}
\newcommand{\Eea} {\end{eqnarray*}}

\newcommand{\al} {\alpha}
\newcommand{\ba} {\beta}
\newcommand{\de} {\delta}

\newcommand{\De} {\Delta}
\newcommand{\la} {\lambda}

\newcommand{\na} {\nabla}

\newcommand{\var} {\varepsilon}

\title{Sign changing solutions of the Brezis-Nirenberg problem in the Hyperbolic space.}

\author{Debdip Ganguly and K. Sandeep\footnote{TIFR Centre for Applicable Mathematics, Post Bag No. 6503
 Sharadanagar,Chikkabommasandra, Bangalore 560065. Email:debdip@math.tifrbng.res.in; sandeep@math.tifrbng.res.in}}

\date{}
\begin{document}

\maketitle
\begin{abstract} In this article we will study the existence and nonexistence of sign changing solutions
for the Brezis-Nirenberg type problem in the Hyperbolic space. We will also establish sharp asymptotic estimates for the solutions and 
the compactness properties of solutions.

\end{abstract}
\section{Introduction}

In this article we will study the equation
\begin{equation} \label{E:1.1}
 -\De_{\bn} u - \la u = |u|^{2^{*}-2} u,   u \in H^{1}(\bn)
\end{equation}
where $\la < (\frac{N-1}{2})^{2}$ and $H^{1}(\bn)$ denotes the Sobolev
space on the disc model of the Hyperbolic space $\bn$, $\De_{\bn}$ denotes the Laplace Beltrami operator on $\bn$ 
(see the Appendix for definitions) and $2^* = \frac{2N}{N-2}$
is the critical Sobolev exponent and $N\ge 3$. \\\\
Though equation \eqref{E:1.1} is a natural generalization of the well known Brezis-Nirenberg equation (\cite{BN})
to the Hyperbolic space, it came to prominense with the discovery of its connection with various other
equations like Hardy-Sobolev-Mazya equations(\cite{CFMS1},\cite{CFMS2},\cite{MS}) and Grushin equations(\cite{B}).
Existence and uniqueness of positive finite energy solutions to \eqref{E:1.1} has been thoroughly investigated
in \cite{MS}, in fact for the general nonlinearity $|u|^{p-2} u$ with $2<p\le \frac{2N}{N-2}$ for $N\ge 3$ and $p>1$ for N=2.
 It is shown in \cite{MS}
that \eqref{E:1.1} has a positive solution iff $\frac{N(N-2)}{4}<\la < (\frac{N-1}{2})^{2}, N\ge 4$ and the solution is unique 
up to hyperbolic isometry. The problem also exhibits a low dimensional phenomenon(nonexistence of positive solution 
for $N=3$ for any $\la$ ), which also implies that the best constant in the Sobolev inequality in the 3-dimensional hyperbolic space
is the same as the corresponding one in the Euclidean space (\cite{BFL}). 
Existence and nonexistence of positive solutions to the above problem in geodesic balls of the hyperbolic space have been 
studied in \cite{S}.
 
In this article we mainly dicuss the sign changing solutions of \eqref{E:1.1}. In the subcritical case, i.e.,when the nonlinear term is 
$|u|^{p-2} u$ with $2<p< \frac{2N}{N-2}$, the problem admits infinitely many sign changing solutions (\cite{BS-2}) for any
$\la < (\frac{N-1}{2})^{2}$. It is also shown in \cite{BS-2} that \eqref{E:1.1} has a pair of radial sign changing solution when $N\ge 7.$
Radial sign changing solutions of \eqref{E:1.1} without the finite energy assumption for the case $\la=0$ has been studied in \cite{BGG}.
So many questions remains unanswered in the critical case. First of all is the restriction on $\la$ for the existence of a positive solution
 is required for the existence of sign changing solutions as well ? One may expect so as the condition is coming from a Pohozaev obstruction
which is applicable to sign changing solutions as well. However we can not apply directly the Pohozaev identity as we do not know the 
behaviour of solutions near infinity. We establish asymptotic estimates for the solutions (see Theorem \ref{T:2.1}) and prove:
\begin{theorem}\label{T:1.1}
 The Eq.(\ref{E:1.1}) does not have a solution if $\la \leq \frac{N(N-2)}{4}$.
\end{theorem}

There has been an extensive study of the Brezis-Nirenberg problem in the past two decades in bounded domains of the Euclidean space
and also on compact Riemannian manifolds(see \cite{So},\cite{Sz}, \cite{JV} and
 the references therein). One of the important result obtained is the existence of infinitely many sign changing solutions
when the dimension $N\geq 7$ (\cite{So},\cite{Sz}). In all these approaches one of the main tool used is the compactness of the 
Brezis-Nirenberg problem established by Solimini(\cite{So}) in higher dimensions. \\
In the hyperbolic case, we prove the following compactness theorem for radial solutions:
\begin{theorem} 
 Let $N \geq 7$ and $\cal{A}$ be a bounded subset of $H^{1}(\bn)$ consisting of radial solutions of \eqref{E:5.2}
for a fixed $\lambda$ and $p$ varying in $(2,2^*]$, then there exists a constant $C$ depending only on $\cal{A}$
such that
\begin{equation}
  |u(x)| \leq C (1-|x|^{2})^{\frac{N-1}{2}}
\end{equation} 
holds for all $u \in \cal{A}.$ 
\end{theorem}
With the help of above theorem we prove:
\begin{theorem}\label{T:1.3}
The Eq.(\ref{E:1.1})has infinitely many non-trivial radial sign changing solutions if
 $N \geq 7$ and $\frac{N(N-2)}{4} < \la < (\frac{N-1}{2})^{2}$.
\end{theorem}

We divide the paper in to four sections. In Section 2, we will prove the asymptotic estimates on the solutions, Section 3
is devoted to the compactness properties, Sections 4 and 5 will respectively prove the nonexistence and existence results and in section
6, we recall the definitions and embeddings of Sobolev spaces on the hyperbolic space.\\
{\bf Notations.} We will denote by $H^{1}(\bn)$ the Sobolev space with respect to the hyperbolic metric and 
$H^{1}_0(\bn)$ will denote the Euclidean Sobolev space on the unit disc. We will denote the hyperbolic volume by $dV_{\bn}$.

\section{Asymptotic estimates}
From the standard elliptic theory we know that the solutions of (\ref{E:1.1}) are in
$C^{2}(\bn)$. But we do not have any information 
on the nature of solutions as $x\rightarrow \I$ (equivalently as $|x| \rightarrow
1$). If \emph{u} is a positive solution of (\ref{E:1.1})
,by moving plane method \emph{u} is radial with respect to a point and the exact
behaviour of $\emph{u}$ as $x \rightarrow \I$ has been obtained 
in \cite{MS} by analysing the corresponding ode. But there is no reason to expect every solution
to be radial(especially the sign changing ones) and hence the above mentioned approach does not help in finding 
apriori estimates in the general case. In this section we will prove the following asymptotic estimate
which plays a major role in the proof of Theorem \ref{T:1.1}.\\
\begin{theorem}\label{T:2.1}
 Let u be a solution of (\ref{E:1.1}),then $|u(x)| + |\grad_{\bn} u(x)|^{2}
\rightarrow 0$ as $x \rightarrow \I$ in $\bn$. If 
$\la \leq \frac{N(N-2)}{4} $ then
 $|u(x)| \leq C \left(\frac{1 -|x|^{2}}{2}\right)^{c_\la}$ where $c_\la = \min
\{\frac{(N-1)+\sqrt{(N-1)^{2}- 4\la}}{2},\frac{N+2}{2}\}$.
\end{theorem}

We will prove this theorem in several steps. First a few propositions.
\begin{proposition}\label{P:2.2}
Let $v \in D^{1,2}(\rn_+)$ be a weak solution of the equation
\begin{equation}\label{eqnlinear}
- \De v +\eta \frac{v}{x_{N}^{2}} = (f)_{x_i} +gv 
\end{equation}
where $\eta \geq 0,f\in L^{\I}_{\mbox{loc}}(\overline{\rn_{+}})$ and $g\in L^{q}_{\mbox{loc}}(\overline{\rn_{+}})$ for
some $q>\frac{n}{2}$,then $v\in L^{\I}_{\mbox{loc}}(\overline{\rn_{+}})$.
\end{proposition}
\begin{remark} Note that in the above proposition $f$ is only assumed to be in
$L^{\I}_{\mbox{loc}}$, by a weak solution
 we mean $v$ satisfies
 \begin{equation}\label{weak-sol}
\int\limits_{\rn_{+}} \nabla v\nabla \phi + \int\limits_{\rn_{+}}\eta
\frac{v\phi}{x_{N}^{2}} = 
-\int\limits_{\rn_{+}}f\phi_{x_i}+ \int\limits_{\rn_{+}}gv\phi\; ,\; \forall \phi
\in C_c^{\I}(\rn_{+})
 \end{equation}

\end{remark}

\begin{proof} 
We will prove the theorem using Moser Iteration. Fix a point $x_{0}\in \partial
\rn_{+}$ and $R>0$.
Define 
$\tilde v = v^{+} +1$ and 
\[
v_{m} =
 \begin{cases}
\tilde v  &\text{if $v < m$}\\
 1 + m  &\text{if $v \geq m$}\\
\end{cases}
\]
For $\ba > 0$ define the test function $w = w_{\ba}$ as $w_{\ba} =
\varphi^{2}(v_{m}^{2\ba}\tilde v - 1)$ where $\varphi \in C_{0}^{\I}(\rn)$, $0 \leq
\varphi \leq 1, \varphi \equiv 1 $  in $B(x_{0},r_{i+1})$, supp$\varphi \subseteq
B(x_{0}, r_{i})$, $R < r_{i+1} < r_{i} < 2R$ and $|\grad \varphi| \leq
\frac{C}{r_{i} - r_{i+1}}$ where $C$ is independent of $\varphi$.\\
Then $0 \leq w \in H_{0}^{1}(\bn)$ and using $\varphi = w$ as the test function in
\eqref{weak-sol}, we get  
\begin{equation} \label{weak-eqn}
 \int\limits_{\rn_{+}} \nabla v\nabla w + \int\limits_{\rn_{+}}\eta
\frac{vw}{x_{N}^{2}} = 
-\int\limits_{\rn_{+}}fw_{x_i}+ \int\limits_{\rn_{+}}gvw
\end{equation}
Now substituting $w$ and observing that $\int\limits_{\rn_{+}}\eta
\frac{vw}{x_{N}^{2}}\geq 0$ we get\\
L.H.S $\geq $
\[
 \int_{\rn_{+}} [v_{m}^{2\ba} \varphi^{2}\grad v \grad \tilde v + 2\ba v_{m}^{2\ba
-1} \tilde v \varphi^{2} \grad \tilde v
\grad v_{m} + 2 \varphi (v_{m}^{2\ba}\tilde v - 1) \grad \tilde v \grad  \varphi] dx
\] \\
In the support of 1st integral $\grad v = \grad \tilde v$, and in the support of 2nd
integral 
$v_{m} = \tilde v, 
\grad v_{m} = \grad \tilde v$.Therefore using Cauchy- Schwartz along with the above
fact we get
\begin{eqnarray}\label{E:2.3}
 \mbox{L.H.S} &\geq & \frac{1}{2} \int_{\rn_{+}}v_{m}^{2\ba} \varphi^{2}|\grad
\tilde v|^{2} dx + 2\ba\int_{\rn_{+}}v_{m}^{2\ba} \varphi^{2}
|\grad v_{m}|^{2} dx\nonumber\\
&&-2 \int_{\rn_{+} }|\grad \varphi|^{2} v_{m}^{2\ba} \tilde v^{2} dx
\end{eqnarray}
The RHS of \eqref{weak-eqn} can be estimated by
$$
  |\int\limits_{\rn_{+}}f[2(v_{m}^{2\ba}\tilde v - 1)\varphi\varphi_{x_i}+
 \varphi^{2}v_{m}^{2\ba}(\tilde v)_{x_i} + 2\ba\varphi^{2}v_{m}^{2\ba-1}\tilde
v(v_m)_{x_i}]| +|\int\limits_{\rn_{+}}g \varphi^{2}v_{m}^{2\ba}\tilde v ^2|
$$
$$\le \frac{1}{4} \int_{\rn_{+}}v_{m}^{2\ba} \varphi^{2}|\grad \tilde v|^{2} dx +
\ba\int_{\rn_{+}}v_{m}^{2\ba} \varphi^{2}
|\grad v_{m}|^{2} dx + C\int_{\rn_{+} }|\grad \varphi| v_{m}^{2\ba} \tilde v dx$$
$$ +C(1+\ba)\int_{\rn_{+}}v_{m}^{2\ba} \varphi^{2}+\int\limits_{\rn_{+}}|g|
\varphi^{2}v_{m}^{2\ba}\tilde v ^2 $$
where $C$ is a constant depending on the $L^{\I}$ norm of $f$ on $B(x_0,2R).$ Since
$\tilde v \ge 1$ we can estimate the RHS as
\begin{eqnarray}\label{rhs}
\mbox{R.H.S} \le \frac{1}{4} \int_{\rn_{+}}v_{m}^{2\ba} \varphi^{2}|\grad \tilde
v|^{2} dx + \ba\int_{\rn_{+}}v_{m}^{2\ba} \varphi^{2}|\grad v_{m}|^{2} dx \nonumber
\\
+C\frac{\ba}{(r_i-r_{i+1})^2} \int_{\rn_{+}} v_{m}^{2\ba} \tilde v^{2} dx
+\int\limits_{\rn_{+}}|g| \varphi^{2}v_{m}^{2\ba}\tilde v ^2
\end{eqnarray}
Using the estimates \eqref{E:2.3} and \eqref{rhs} in \eqref{weak-eqn} we get
\begin{eqnarray}\label{weak-esti}
\frac{1}{2} \int_{\rn_{+}}v_{m}^{2\ba} \varphi^{2}|\grad \tilde v|^{2} dx +
2\ba\int_{\rn_{+}}v_{m}^{2\ba} \varphi^{2}|\grad v_{m}|^{2} dx \nonumber\\
\le C\frac{1+\ba}{(r_i-r_{i+1})^2} \int_{\rn_{+}} v_{m}^{2\ba} \tilde v^{2} dx
+ \int\limits_{\rn_{+}}|g| \varphi^{2}v_{m}^{2\ba}\tilde v ^2
\end{eqnarray}

Defining  $\bar w = v_{m}^{\ba}\tilde v $, \eqref{weak-esti} becomes
\begin{eqnarray}\label{estimate}
 \frac{1}{4(1 + 2\ba)} \int_{\rn_{+}}|\grad(\varphi \bar w)|^{2} dx \leq 
 C\frac{(1+\ba)}{(r_i-r_{i+1})^2} \int_{\rn_{+}}(\varphi \bar w )^{2} dx
+ \int\limits_{\rn_{+}}|g| (\varphi \bar w) ^2 \nonumber \\
\le C\frac{1+\ba}{(r_i-r_{i+1})^2} \int_{\rn_{+}}(\varphi \bar w )^{2} dx +
C \left(\int\limits_{B(x_0,r_{i+1})}|\varphi \bar
w|^{q^\prime}\right)^{\frac{2}{q^\prime}}
\end{eqnarray}
where $\frac{1}{q} + \frac{1}{q^{\prime}} = 1$.Since $q > \frac{N}{2} \Rightarrow
{q}^{\prime} < \frac{N}{N-2} = r$ (say).Now let 
$\frac{1}{{q}^{\prime}} = \theta + \frac{1-\theta}{r}$,then using interpolation
inequality we get 
\[
 |(\varphi \bar w)^{2}|_{L^{{q}'}} \leq \var(1 - \theta)|(\varphi \bar
w)^{2}|_{L^{r}} + c_{1} \var^{-\frac{1-\theta}{\theta}}
|(\varphi \bar w)^{2}|_{L^{1}}\;,\forall \var
\]
where $\theta$ depends on $N, {q}^{\prime}$.Note that $2r = 2^{*}$.Therefore 
\[
 |(\varphi \bar w)^{2}|_{L^{r}} = |\varphi \bar w|^{2}_{L^{2^{*}}(\rn_{+})} \leq C
|\grad (\varphi \bar w)|^{2}_{L^{2}(\rn_{+})}
\]
Hence 
\[
 |(\varphi \bar w)^{2}|_{L^{q'}} \leq C\var |\grad(\varphi \bar
w)|^{2}_{L^{2}(\rn_{+})} + C\var ^{-\al} |(\varphi \bar w)^{2}|_{L^{1}{\rn_{+}}} 
\]
Now choosing $\var$ suitably and substituting in \eqref{estimate}, we get
\[
 \int_{\rn_{+}} |\grad (\varphi \bar w)|^{2} dx \leq \frac{C(1 + \ba)^\al}{(r_{i}
-r_{i +1})^{2}} \int_{B(x_{0},r_{i})} \bar w^{2} dx
\] Now using the Sobolev inequality in the above expression we get
\[
 \left(\int_{B(x_{0},r_{i+1})} \bar w^{\frac{2N}{N-2}} dx \right)^{\frac{N-2}{N}} \leq
\frac{C(1 + \ba)^{\al}}{(r_{i} -r_{i +1})^{2}} 
\int_{B(x_{0},r_{i})} \bar w^{2} dx
\]
Now using $\chi = \frac{N}{N-2} > 1 $,$\bar w = v_{m}^{\ba}\tilde v$, $v_{m} \leq
\tilde v$ and $\gamma = 2(\ba + 1)$ we get 
\[
 \left( \int_{B(x_{0},r_{i+1})} v^{\gamma \chi}_{m} dx \right)^{\frac{1}{\gamma \chi}}
\leq \left[\frac{C(1 + \ba)^{\al}}{(r_{i} -r_{i +1})^{2}}
\right]^{\frac{1}{\gamma}}
\left( \int_{B(x_{0},r_{i})} \tilde v^{\gamma} dx \right)^{\frac{1}{\gamma}}
\]
Now letting $m \rightarrow \I$ we get
\[
  \left( \int_{B(x_{0},r_{i+1})} \tilde v^{\gamma \chi} dx \right)^{\frac{1}{\gamma
\chi}} \leq \left[\frac{C(1 + \ba)}{(r_{i} -r_{i +1})^{2}} 
\right]^{\frac{1}{\gamma}}
\left( \int_{B(x_{0},r_{i})} \tilde v^{\gamma} dx \right)^{\frac{1}{\gamma}}
\]
provided $|\tilde v|_{L^{\gamma}(B_{x_{0},r_{i+1}})}$ is finite. C is a positive
constant independent of $\gamma$. Now we will complete the proof 
by iterating the above relation. Let us take $\gamma =2,2\chi, 2\chi^{2}\ldots$
i.e.,$\gamma_{i} = 2\chi^{i}$ for $i = 0,1,2,\ldots 
r_{i+1}= R + \frac{R}{2^{i+1}}$. Hence for $\gamma = \gamma_{i}$ we get 
\[
 \left(\int_{B(x_{0},r_{i+1})} \tilde v^{\gamma_{i}} dx \right)^{\frac{1}{\gamma_{i}}}
\leq C^{\frac{i+1}{\chi^{i+1}}}
\left(\int_{B(x_{0},r_{i})} \tilde v^{\gamma_{i+1}} dx
\right)^{\frac{1}{\gamma_{i+1}}}
\] 
where $C>1$ is a constant depends on $ R, N, ||g||_{L^{q}(B(x_0,2R))} $ 

and $ ||f||_{L^{\I}(B(x_0,2R))} $. Now by iteration we obtain 
\[
 \left(\int_{B(x_{0},r_{i+1})} \tilde v^{\gamma_{i}} dx \right)^{\frac{1}{\gamma_{i}}}
\leq C^{\sum {\frac{i+1}{\chi^{i+1}}}}
\left(\int_{B(x_{0},r_{i})} \tilde v^{\gamma_{i+1}} dx
\right)^{\frac{1}{\gamma_{i+1}}}
\]
letting $i\rightarrow \I$ we obtain 
\[
 \mbox{sup}_{B(x_{0} ,R)} \tilde v \leq \tilde C |\tilde v|_{L^{2}(B(x_{0},2R))}.
\]
This proves the local boundedness of $u^+$. Similarly we get the boundedness of
$u^-$ as $-u$ also
satisfies the same equation with $-f$ in place of $f$.  This proves the theorem.
\end{proof}

\begin{proposition}\label{h2reg}
Let $u\in D^{1,2}(\rn_+)$ be a weak solution of the problem
\begin{equation}\label{linearized}
- \De u +\eta \frac{u}{x_{N}^{2}} = |u|^{2^{*}-2}u 
\end{equation}
with $\eta \ge 0,$ then $u_{x_i},u_{x_{i}x_{i}}\in D^{1,2}(\rn_+)$ for all $1\<i<N.$
\end{proposition}

\begin{proof} Using Moser iteration as in Brezis-Kato (See \cite{St}, Appendix B, Lemma B3 )
we can show that $ u\in L^q_{\rm loc} (\overline{\rn_+})$ for some $q>\frac{2N}{N-2}.$ Thus from 
Proposition \ref{P:2.2} with $f=0$ and $g =|u|^{2^{*}-2} $ we get $|u(x)|\le M$ for all $|x|\le 1$
for some $M>0.$ Since the Kelvin transform of $u$ given by $\frac{1}{|x|^{N-2}}u(\frac{x}{|x|^2})$ also 
satisfies \eqref{linearized} we get $u\in L^\I(\rn_+).$\\
We will show that $u_{x_i},u_{x_{i}x_{i}}\in D^{1,2}(\rn_+)$ by the method of difference quotients. 
 The case of $u_{x_i}$ follows exactly as in Theorem 2.1 of \cite{BS}. To prove the estimate on $u_{x_{i}x_{i}} $,
 first note that by standard elliptic theorey $u \in C^{3,\al}(\rn_+)$. Differentiating \eqref{linearized} with 
respect to $x_i$ we see that $u_{x_i} \in D^{1,2}(\rn_+)$ satisfies
\begin{equation}\label{linearized-d}
- \De u_{x_i} +\eta \frac{u_{x_i}}{x_{N}^{2}} = (2^{*}-1)|u|^{2^{*}-2}u_{x_i} 
\end{equation}
Thus,we have for any $w\in D^{1,2}(\rn_{+})$ 
\begin{equation}\label{E2g}
 \int \grad u_{x_{i}} \grad w + \eta \int \frac{u_{x_{i}} w}{x_{N}^{2}} = (2^{*}-1)\int
|u|^{2^{*}-2} u_{x_{i}} w
\end{equation}
For $|h|> 0$ and $i< N$, define $w = -D_{i}^{-h}(D_{i}^{h}u_{x_{i}})$ where
$D_{i}^{h}$ denotes the difference quotient 
\[
 D_{i}^{h} u_{x_{i}}(x) = \frac{u_{x_{i}}(x + he_{i}) - u_{x_{i}}(x)}{h} 
\]

For this choice of $w$ the L.H.S of (\ref{E2g}) simplifies to $\int|\grad (D_{i}^{h} u_{x_{i}})|^{2}
 + \eta \int \frac{(D_{i}^{h} u_{x_{i}})^{2}}{x_{N}^{2}}$ while R.H.S can be
estimated as

\begin{eqnarray*}
 |\int |u|^{2^{*}-2} u_{x_{i}}w| &&= | \int \frac{\partial}{ \partial x_{i}}
D_{i}^{h}(|u|^{2^{*}-2} u) D_{i}^{h} u_{x_{i}}|\\
&&= |-  \int D_{i}^{h} (|u|^{2^{*}-2} u) \frac{\partial}{\partial x_{i}} (D_{i}^{h}
u_{x_{i}})|\\
&& \leq \int |D_{i}^{h}(|u|^{2^{*}-2} u)| |\grad (D_{i}^{h} u_{x_{i}})|\\
&& \leq C \int [|u|^{2^{*}-2}(x + he_{i}) + |u|^{2^{*}-2}(x)] |D_{i}^{h} u| |\grad
(D_{i}^{h} u_{x_{i}})|\\
&& \leq C \int |D_{i}^{h} u| |\grad (D_{i}^{h} u_{x_{i}})|\\
&& \leq C\varepsilon \int |\grad (D_{i}^{h} u_{x_{i}})|^{2} + \frac{C}{4
\varepsilon} \int |D_{i}^{h} u|^{2}
\end{eqnarray*}
 By choosing $C\varepsilon < 1$ we have  
\[
 \int_{\rn_{+}} |\grad (D_{i}^{h} u_{x_{i}})|^{2} \leq C \int_{\rn_{+}} |\grad
u|^{2} \leq C
\]
where $C$ is independent of $h$ and this implies $\int
|\grad u_{x_{i}x_{i}}| \leq M$ and this completes the proof.
\end{proof}

\begin{proposition}\label{ode-estimate}
Let $f : (0,\I)\rightarrow \R$ be a continuous function bounded in $(0,1)$ ,$\eta>0$
be a positive constant and $v$ solves the ODE 
\begin{equation}\label{E:2.8}
  -\frac{d^{2} v}{d r^{2}} + \eta \frac{v}{r^{2}} = f(r),\;\; v(0)=0
\end{equation}
then there exist constants $C_1,C_2$ depending only on the $L^{\I}$ norm of
$f|_{(0,1)}$ such that
\begin{equation}
|v(r)| \le C_1 r^{\frac{1 + \sqrt{4 \eta + 1}}{2}} + C_{2} r^{2} 
\end{equation}
holds for all $r \in (0,1).$
\end{proposition}.
\begin{proof} Let $v(r) =\theta(log r)$ then the equation  (\ref{E:2.8}) transforms to
\begin{equation}\label{E:2.9}
 \frac{d^{2} \theta}{d t^{2}}(t) - \frac{d \theta}{d t}(t) - \eta \theta(t) = - e^{2
t} f(e^{t})
\end{equation}
\[
 \theta(t)\rightarrow 0 \quad \mbox{as}\quad  t\rightarrow - \I
\]

Using the method of variation of parameters we can write 
\[
 \theta(t) = \theta_{c}(t) + \theta_{p}(t)
\]
where $\theta_{c}(t)$ is the Complementary Function given by 
\begin{equation}\label{E:2.10}
 \theta_{c}(t) =  C_{1} e^{m_{1}t} + C_{2} e^{m_{2} t}
\end{equation}
with $m_{1} = \frac{1+ \sqrt{4  \eta + 1}}{2}$ and  $m_{2} = \frac{1- \sqrt{4 \eta +
1}}{2}$,\\ $\theta_{p}$ is a particular integral given by 
\begin{equation}\label{E:2.12}
 \theta_{p}(t) = v_{1}(t)e^{m_{1}t} + v_{2}(t)e^{m_{2}t}
\end{equation}
where
\[
 v_{1}(t) = v_{1}(0) + \frac{1}{\sqrt{4 \eta+ 1}} \int_{t}^{0} e^{(m_{2} + 1)s}
f(e^{s}) ds
\]
\[
  v_{2}(t) = \frac{1}{\sqrt{4 \eta + 1}} \int_{-\I}^{t} e^{(m_{1} + 1)s} f(e^{s}) ds
\]
From the expressions for $v_1$ we get 
\begin{equation}\label{E:2.13}
 v_{1}(t) \leq v_{1}(0) + \frac{1}{\sqrt{4\eta + 1}} \frac{M}{m_{2} + 1} -
 \frac{M}{\sqrt{4 \eta + 1}(m_{2} + 1)} e^{(m_{2} + 1)t}
\end{equation}
where $|f|\le M $ on $(0,1).$ Thus,
\begin{equation}\label{E:2.14}
|v_{1}(t) e^{m_{1}t}| \leq C_1e^{m_{1}t} +  C_2 e^{2t}
\end{equation}

Similarly from the expression of $v_{2}(t)$ we have 
\begin{equation}\label{E:2.16}
|v_{2}(t)| \leq \frac{M}{\sqrt{4 \eta +1}} \int_{-\I}^{t} e^{(m_{1} + 1)s} ds = 
 \frac{M}{\sqrt{4 \eta +1}(m_{1} + 1)} e^{(m_{1} +1)t}
\end{equation}
Thus ,
\begin{equation}\label{E:2.18}
 |v_{2}(t) e^{m_{2}t}| \leq \frac{M}{\sqrt{4 \eta +1}(m_{1} + 1)} e^{2t}
\end{equation}
Since $\theta(t)\rightarrow 0 \quad \mbox{as}\quad  t\rightarrow - \I$,
 using (\ref{E:2.14})and (\ref{E:2.18}) we get  $C_{2} = 0$ in \eqref{E:2.10}. \\
Using these informations we get
\begin{equation}\label{E:2.23}
 |\theta(t)|  \leq C_{1}e^{m_{1}t} + \tilde C_{2} e^{2t}
\end{equation}
for all $t<0$. Changing the variable as $r=e^t$ proves the proposition.
\end{proof}

{\bf Proof of Theorem \ref{T:2.1}} Let M be the isometry between $\bn$ and the upper
half space model $\hn $ given by 
\begin{equation}\label{iso}
 M(x) := \left( \frac {2 x^\prime}{(1+x_N)^2 + |x^\prime|^2},\frac {1-|x|^2}{(1+x_N)^2 + |x^\prime|^2} \right)
\end{equation}
where a point in $\hn $ is denoted by $x=(x^\prime,x_N)\in \R^{N-1}\times \R.$
 Then $\tilde u = u\circ M : \hn \rightarrow \R$ satisfies the equation (note $M^{-1} =M$)
\begin{equation}\label{hamodelen}
-\De_{\hn} \tilde u - \la \tilde u = |\tilde u|^{2^{*}-2} \tilde u,   \tilde u \in
H^{1}(\hn)
\end{equation}
where $\De_{\hn}$ is the Laplace-Beltrami operator in $\hn$ given by
\begin{equation}
 \De_{\hn}u =x_N^2\De u -(N-2)x_N u_{x_N}
\end{equation}
Making a conformal change of the metric, defining
$v(x)= x_N^{-\frac{N-2}{2}}\tilde u(x)$, $v$ satisfies the equation \eqref{linearized}
with $\eta= (\frac{N(N-2)}{4}-\la)\ge 0. $\\
 From standard elliptic theory we know that $v \in C^{3,\al}_{\rm loc}(\rn_+)$.
Moreover Proposition \ref{h2reg} tells us that  $v_{x_i},v_{x_ix_i} \in D^{1,2}(\rn_+)$ for all $1\<i<N$
 and $v$ is bounded.
 Next we claim that $ v_{x_i}$ and $v_{x_ix_i}$ are locally bounded for $1<i<N.$\\\\
We know that $ v_{x_i}$ satisfies \eqref{linearized-d}. Applying Proposition \ref{P:2.2}
with with $f=0,g=(2^*-1)|v|^{2^*-2}$, we get $v_{x_i}$ is locally bounded.\\
Since $v_{x_i}$ is in $D^{1,2}({\rn}_+)$ we get
$$\int\limits_{\rn_+}\nabla v_{x_i} \nabla \varphi +\eta
\int\limits_{\rn_+}\frac{v_{x_i} \varphi}{x_N^2} =
\int\limits_{\rn_+}(2^*-1)|v|^{2^*-2}v_{x_i}\varphi $$
for all $\varphi \in C_c^{\I}(\rn_+)$. Taking $\varphi_{x_i}$ instead of $\varphi$
and an integration by parts gives
$$\int\limits_{\rn_+}\nabla v_{x_ix_i} \nabla \varphi +\eta
\int\limits_{\rn_+}\frac{v_{x_ix_i} \varphi}{x_N^2} =
-(2^*-1)\int\limits_{\rn_+}|v|^{2^*-2}v_{x_i}\varphi_{x_i} $$
This shows that $v_{x_ix_i}$ satifies \eqref{eqnlinear} with
$f=(2^*-1)|v|^{2^*-2}v_{x_i}$ and $g=0.$
This proves the local boundedness of $v_{x_ix_i}$ for $i<N.$\\\\
Now we will estimate the solution $v$. Fix a point
$x'=(x_{1},\ldots,x_{N-1})\in \mathbb{R}^{N-1}$, and define $v(r) = v(x',r)$ for
$r>0.$ Then $v$ satisfies the ODE \eqref{E:2.8}
with $f(r) = \sum\limits_{i<N}v_{x_ix_i}(x',r)+ |v|^{2^*-2}v(x',r).$ Thus from
Proposition \ref{ode-estimate}
we get
\begin{equation}\label{E:2.26}
 |v(x',r)| \leq C_{1} r^{\frac{1 + \sqrt{4 \eta + 1}}{2}} + C_{2} r^{2}  
\end{equation}
 Since $C_{1}$ and $C_{2}$ depends only on the local bound on $f$, from the uniform
bound of $\sum\limits_{i<N}v_{x_ix_i}(x',r)+ |v|^{2^*-2}v(x',r)$ on compact subsets
of $\overline{\rn_+}$ we get the above estimate locally in $\overline{\rn_+}$, in
particular
\begin{equation}\label{E:2.27}
  |v(x)| \leq C_{1} {x_N}^{\frac{1 + \sqrt{4 \eta + 1}}{2}} + C_{2} x_N^{2}  \quad
\forall x\in B(0,1)\cap \rn_{+}
\end{equation}
To get a global bound on $v$, first observe that if $v$ is a solution of
(\ref{linearized}) 
then its Kelvin transform $\tilde v(x) := \frac{1}{|x|^{N-2}} v(\frac{x}{|x|^{2}})$
also solves  (\ref{linearized}).So $\tilde v$ also satisfies the estimate
\eqref{E:2.27}.
 Hence we have,
\begin{equation}\label{E:2.29}
|v(x)| \leq C_{1} \frac{x_N^{m_{1}}}{|x|^{N-2 + 2m_{1}}} + C_{2}
\frac{x_N^{2}}{|x|^{N-2 +4}} \quad \forall x \in (B(0,1))^{c}\cap \rn_{+} 
\end{equation}
So, combining (\ref{E:2.27}) and (\ref{E:2.29}) we have,
\begin{equation}\label{E:2.30}
|v(x)|  \leq C_{1} \frac{x_N^{m_{1}}}{(1 + |x|^{2})^{\frac{N-2+2m_{1}}{2}}} + C_{2}
\frac{x_N^{2}}{(1 + |x|^{2})^{\frac{N-2+4}{2}}}
 \quad \forall x\in  \rn_{+}
\end{equation}
Now recall that $\tilde u = x_N^{\frac{N-2}{2}} v$  and hence $\tilde u$ satisfies the
estimate   
\begin{equation}\label{E:2.31}
|u(x)| \leq   C_{1} \frac{x_N^{m_{1}+(N-2)/2}}{((1+x_N)^{2} +
|x'|^{2})^{\frac{N-2+2m_{1}}{2}}} + C_{2} \frac{x_N^{{2} +(N-2)/2}}
{((1+x_N)^{2} + |x'|^{2})^{\frac{N-2+4}{2}}}
\end{equation}
For a point $\xi \in \bn$, let $M(\xi) = x$ then $\frac{1-|\xi|^2}{2} =
\frac{x_N}{(1+x_N)^2 +|x'|^2}$.
Since $u = \tilde u \circ M$ we get 
\begin{equation}\label{E:2.32}
 |u(\xi)| = |u(x',x_N)| \leq   C_{1} \frac{x_N^{m_{1}+(N-2)/2}}{((1+x_N)^{2} +
|x'|^{2})^{\frac{N-2+2m_{1}}{2}}} + C_{2} \frac{x_N^{{2} +(N-2)/2}}
{((1+x_N)^{2} + |x'|^{2})^{\frac{N-2+4}{2}}}
\end{equation}
where $|\xi| < 1$. Now putting the value of $m_{1}$ and $m_{2}$ we get 
\[
  |u(\xi)| \leq C\left[ \left(\frac{1 - |\xi|^{2}}{2}\right)^{\frac{(N-1) +
\sqrt{(N-1)^{2}- 4\la}}{2}} + \left(\frac{1-|\xi|^{2}}{2}\right)^
{\frac{N+2}{2}}\right]
\]
This proves the theorem.
\section{Compactness of solutions} 
In this section we will study the compactness properties of solutions of the equation
\begin{equation}\label{E:5.2}
 - \Delta_{\bn} u -\lambda u = |u|^{p-2}u, \quad u \in H^{1}_{r}(\bn)
\end{equation}
where $H^{1}_{r}(\bn)$ denotes the subspace of $H^{1}(\bn)$ consisting of radial functions, $p\in (2, 2^{*}]$ 
and $\frac{N(N-2)}{4}< \la <\left(\frac{N-2}{2}\right)^2$.\\
First recall the following radial estimate (see \cite{BS-2},Theorem 3.1 for a proof): 
\begin{lemma}\label{T:5.1}
 Let $\cal{A}$ be a bounded subset of $H^{1}_r(\bn)$ , then there exists a constant $C$ depending only on $\cal{A}$
such that
\begin{equation}
  |u(x)| \leq C|x|^{-\frac{N}{2}} \left( 1-|x|^{2}\right)^{\frac{N-1}{2}}
\end{equation} 
holds for all $u \in \cal{A}.$ 
\end{lemma}
This estimate gives us control over the radial functions away from the origin. The main result we prove in this section rules out blow-up
at the orgin in higher dimensions if members of $\cal{A}$ are  in addition solutions of \eqref{E:5.2}.\\    
\begin{theorem} \label{T:5.2}
 Let $N \geq 7$ and $\cal{A}$ be a bounded subset of $H^{1}_{r}(\bn)$ consisting of solutions of \eqref{E:5.2}
for a fixed $\lambda$ and $p$ varying in $(2,2^*]$, then there exists a constant $C$ depending only on $\cal{A}$
such that
\begin{equation}
  |u(x)| \leq C (1-|x|^{2})^{\frac{N-1}{2}}
\end{equation} 
holds for all $u \in \cal{A}.$ 
\end{theorem}
As a corollory we have the following compactness theorem :
\begin{corollary}\label{comp-sol} Let $N \geq 7$ and $u_n $ be a sequence of solutions of \eqref{E:5.2} with $p=p_n \in (2,2^*]$. 
Suppose $p_n\rightarrow p_0 \in (2,2^*]$ and $u_n$ is bounded in $H^{1}_{r}(\bn)$, then up to a subsequence $u_n \rightarrow u $ in 
$H^{1}_{r}(\bn)$ 
and $u$ solves \eqref{E:5.2} with $p=p_0$. Moreover $u_n \rightarrow u $ in $C(\bn)$.
\end{corollary}
\begin{proof} Since $u_n$ is bounded in $H^{1}_{r}(\bn)$ up to a subsequence we may assume that $u_n$ converges weakly and
pointwise a.e. to $u\in H^{1}_{r}(\bn).$ We can immediately see that $u $ solves \eqref{E:5.2} with $p=p_0$ and hence
$$ \int\limits_{\bn}|\nabla_{\bn} u|^2 dV_{\bn} - \la \int\limits_{\bn}u^2 dV_{\bn} = \int\limits_{\bn}|u|^{p_0}dV_{\bn} $$
Since $u_n $ solves \eqref{E:5.2} with $p=p_n $ we get
$$ \int\limits_{\bn}|\nabla_{\bn} u_n|^2 dV_{\bn} - \la \int\limits_{\bn}u_n^2 dV_{\bn} = \int\limits_{\bn}|u_n|^{p_n}dV_{\bn} $$
Using the estimate in Theorem \ref{T:5.2} and dominated convergence theorem we get the RHS converges to $\int\limits_{\bn}|u|^{p_0}dV_{\bn}.$
Combining  we get $$ \int\limits_{\bn}|\nabla_{\bn} u_n|^2 dV_{\bn} - \la \int\limits_{\bn}u_n^2 dV_{\bn} \rightarrow 
 \int\limits_{\bn}|\nabla_{\bn} u|^2 dV_{\bn} - \la \int\limits_{\bn}u^2 dV_{\bn} $$ and hence in $H^{1}_{r}(\bn)$ thanks to Lemma \ref{PSE-l} .
Now the convergence in $C(\bn)$ follows by standard elliptic estimates and the decay estimate.
\end{proof}

{\bf Proof of Theorem \ref{T:5.2}} Suppose the theorem is not true, then there exists $u_n \in \cal{A}$ such that 
$\max\limits_{x\in \bn}|u_n(x)| = |u_n(x_n)|\rightarrow \infty$ where $u_n$ satisfies \eqref{E:5.2} with $p =p_n$ and we
assume $p_n \rightarrow p_0 \in (2,2^*].$ From Lemma \ref{T:5.1}, it is clear that $x_n\rightarrow 0.$
We will show that this leads to a contradiction.\\\\
Define, $v_n(x) = \left(\frac{2}{1-|x|^{2}}\right)^{\frac{N-2}{2}} u_n$ , then $v_n$ is a bounded sequence in the Euclidean Sobolev space
$H^1_0(\bn)$ and solves the Euclidean equation 
\begin{equation}\label{conc-eu}
 -\Delta v_n - \tilde \lambda \left(\frac{2}{1- |x|^{2}}\right)^{2} v_n = |v_n|^{p_n-2}v_n\left(\frac{2}{1- |x|^{2}}\right)^{q_n},
v_n\in H^{1}_{0}(\bn)
\end{equation}
where $\tilde \lambda = \la -\frac{N(N-2)}{4}>0$ and $q_n= \frac{2N-p_n(N-2)}{2}.$ Using
 Lemma \ref{T:5.1} 
\begin{equation}\label{est-ra}
  |v_n(x)| \leq C|x|^{-\frac{N}{2}} \left( 1-|x|^{2}\right)^{\frac{1}{2}},\;\; \forall n
\end{equation}
Also if $p_0<2^*$ by standard elliptic estimates we get $\max\limits_{|x|\le \frac12}|v_n(x)| \le C<\infty,\forall n$, which is 
impossible as $|v_n(x_n)|\rightarrow \infty.$ Therefore $p_0=2^*$.\\
Since $v_n$ is bounded in $H^{1}_{0}(\bn)$ we may assume up to a subsequence  $v_n$ converges weakly and pointwise a.e to 
$v$ in $H^{1}_{0}(\bn)$. If this convergence
is strong then by standard Brezis-Kato type arguments we get $\max\limits_{|x|\le \frac12}|v_n(x)| \le C<\infty,\forall n.$ 
Therefore $v_n\not\rightarrow v$ in $H^{1}_{0}(\bn)$. However $v$ solves
\begin{equation}\label{co-lt-eu}
  -\Delta v - \tilde \lambda \left(\frac{2}{1- |x|^{2}}\right)^{2} v = |v|^{2^*-2}v
\end{equation}

Choose a cut off function $\varphi \in C_{0}^{\I}(\rn)$ such that $\varphi =1$ in $B_{1}= \{x\in \bn :
|x|< \frac{1}{2}\}$,  $\varphi = 0$ in $B_{2} = \{x \in \bn : |x| \geq \frac{5}{6}\}$ and $0\leq \varphi \leq 1$.
Let $w_n= \varphi v_n, \tilde w_n= (1-\varphi) v_n $ , then $v_n=w_n + \tilde w_n $.
Multiplying \eqref{conc-eu} by $(1-\varphi)^2 v_n $ and integrating by parts we get 

$$ \int\limits_{\bn}[|\nabla \tilde w_n|^2-\tilde \lambda (\frac{2}{1- |x|^{2}})^{2} |\tilde w_n|^2] =
\int\limits_{\bn} |v_n|^{p_n-2}(\tilde w_n)^2 (\frac{2}{1-|x|^{2}})^{q_n} + \int\limits_{\bn}v_n^2|\nabla \varphi|^2$$ 
Using the estimate \eqref{est-ra} and applying dominated convergence theorem we easily see that the RHS converges to
$\int\limits_{\bn} |v|^{2^*-2}(\tilde w)^2 + \int\limits_{\bn}v^2|\nabla \varphi|^2,$
where $\tilde w= (1-\varphi) v $.\\ 
Multiplying \eqref{co-lt-eu} by $(1-\varphi)^2 v $ and integrating by parts we get 
$$\int\limits_{\bn}\left[|\nabla \tilde w|^2-\tilde \lambda \left(\frac{2}{1- |x|^{2}}\right)^{2} |\tilde w|^2\right] 
=\int\limits_{\bn} |v|^{2^*-2}(\tilde w)^2 + \int\limits_{\bn}v^2|\nabla \varphi|^2 $$
Thus 
$$ \int\limits_{\bn}\left[|\nabla \tilde w_n|^2-\tilde \lambda \left(\frac{2}{1- |x|^{2}}\right)^{2} |\tilde w_n|^2\right]
\rightarrow \int\limits_{\bn}\left[|\nabla \tilde w|^2-\tilde \lambda \left(\frac{2}{1- |x|^{2}}\right)^{2} |\tilde w|^2\right]$$
 and hence $ \tilde w_n \rightarrow  \tilde w$ in $H^{1}_{0}(\bn)$.
Therefore $w_n $ converges weakly to $w = \varphi v $ but not strongly.\\
Also $w_n \in H_{0}^{1}(B_{2}^{c})$ satisfies the equation

\begin{eqnarray*}
 -\Delta w_n - \tilde \lambda (\frac{2}{1- |x|^{2}})^{2} w_n = |v_n|^{p-2}w_n \left(\frac{2}{1-|x|^{2}}\right)^{q_n}
- 2 \langle \grad v_n, \grad \varphi \rangle
 - v_n \Delta \varphi ,
\end{eqnarray*}
Thus proceeding exactly as in Lemma 6.2 of \cite{So}, we see that up to a subsequence $w_n$ is a concentrating sequence. i.e.,there
 exists a positive integer 
$k$ and $\varphi_{i} \in D^{1,2}(\rn),i=1,...k$ satisfying $-\Delta \varphi_{i}=  |\varphi_{i}|^{2^*-2} \varphi_{i}$ , $k$ sequence of
positive real numbers $\epsilon_n^i$ and $y_n^i \in \bn,y_n^i \rightarrow 0$, $i=1,...,k$ such that 
 \begin{equation}\label{E:5.6}
 w_{n} - \sum_{i= 1}^{k}\varphi_{i,n}  \rightarrow w \quad \mbox{in} \quad L^{2^\ast}(\rn) 
\end{equation}
where $\varphi_{i,n}(x) = [\epsilon_{n}^{i}]^{\frac{2-N}{2}}\varphi_{i}(\frac{x-y_n^i}{\epsilon_{n}^{i}})$.
Moreover using (\ref{est-ra}) we see that that $ |w_n|$ 
(extended by zero out of $B_{2}^{c}$) solves 
\begin{equation}\label{E:5.5}
 -\Delta |w_n| \leq b  |w_n|^{2^{*}-1} + A
\end{equation}
in the sense of distributions where $b$ and $A$ are constants independent of $n$. i.e, $ |w_n|$ is a controlled sequence in the sense of
Solimini (\cite{So}).Thus if we let $\epsilon_n =\epsilon_n^i, y_n=y_n^i $ where $i$ is chosen such that 
$\limsup\limits_{n\rightarrow \I} \frac{\epsilon_n^j}{\epsilon_n^i} \le C$ for all $j=1,...,k,$ we have from Proposition 3.1 and 
Corollory 4.1 of \cite{So} 

\begin{lemma}\label{l:5.7}
 Then there exists a constant $C>0$ such that 
\begin{equation}\label{E:5.7}
 \sup\limits_{n\in \mathbb{N}}\;\max\limits_{(C+1)\sqrt{\epsilon_n}\le|x|\le (C+4)\sqrt{\epsilon_n}}|w_{n}(x)| < \infty
\end{equation}
 Moreover there 
exist $t_{n}\in \left[{C+2}, {C+3} \right ] $ such that,
\begin{equation}\label{E:5.8}
 \int_{\partial B_{n}} |\grad w_{n}|^{2} \leq C\epsilon_{n}^{\frac{N-3}{2}}
\end{equation}
where $B_{n} = B(y_n, t_{n}\epsilon_{n}^{1/2})$.
\end{lemma}

With this estimate and a local Phozaev identity we will arrive at a contradiction.
First we will derive the local Pohozaev identity. Let us denote the outward normal to $\partial B_n$ by $\vec{n}$.

Multiplying \eqref{conc-eu} by $x.\grad v_n $, and  integrating by parts over $B_{n}$  we get, 
\begin{eqnarray}\label{E:5.11}
 &&\int_{B_n} \grad v_n.\grad (x.\grad v_n) - \tilde \la \int_{B_n}
 \left(\frac{2}{1-|x|^{2}}\right)^{2} v_n (x.\grad v_n) \nonumber \\
&&= \int_{B_n} |v_n|^{p_n-2} v_n \left( \frac{2}{1-|x|^{2}} \right)^{q_n} 
\end{eqnarray}
The RHS of (\ref{E:5.11}) can be simplified as
\begin{eqnarray}\label{E:5.12}
 &&\frac{1}{p_n}\int_{B_n} (\grad (|v_n|^{p_n}). x)
 \left( \frac{2}{1-|x|^{2}}\right)^{q_n}
= \frac{1}{p_n} \int_{\partial B_n} |v_n|^{p_n} \left(\frac{2}{1-|x|^{2}}\right)^{q_n} (x. \vec{n})\nonumber\\ 
&&- \frac{N}{p_n}\int_{B_n} |v_n|^{p_n} \left( \frac{2}{1-|x|^{2}}\right)^{q_n}
 - \frac{q_n}{p_n} \int_{B_n} |v_n|^{p_n} |x|^{2} \left( \frac{2}{1-|x|^{2}}\right)^{q_n + 1} 
\end{eqnarray}
By direct calculation and integration by parts, LHS of (\ref{E:5.11}) simplifies as 
\begin{eqnarray}\label{E:5.13}
 \mbox{LHS}  &&= -\int_{\partial B_n} (\grad v_n. x)(\grad v_n.\vec{n}) + \frac{1}{2}\int_{\partial B_n} |\grad v_n|^{2} (x.\vec{n}) 
+ \frac{2-N}{2} \int_{B_n} |\grad v_n|^{2}\nonumber\\
&& \frac{\tilde \lambda N}{2}\int_{B_n} \left(\frac{2}{1- |x|^{2}}\right)^{2} v_n^{2} + \tilde \lambda \int_{B_n}
\left(\frac{2}{1- |x|^{2}}\right)^{3} |x|^{2} v_n^{2} \nonumber\\
 &&- \frac{\tilde \lambda}{2} \int_{\partial B_n}  \left(\frac{2}{1- |x|^{2}}\right)^{2}v_n^{2}
(x. \vec{n}) 
\end{eqnarray}

Now from the equation \eqref{conc-eu} we have
\begin{eqnarray}\label{E:5.10}
 &&\int_{B_n} |\grad v_n|^{2} dx - \tilde \lambda \int_{B_n} \left(\frac{2}{1-|x|^{2}}\right)^{2} v_n^{2} dx \nonumber \\
&&= \int_{B_n} |v_n|^{p_n} \left(\frac{2}{1-|x|^{2}}\right)^{q_n} + \int_{\partial B} (\grad v .\vec{n}) v 
\end{eqnarray}

Substituting \eqref{E:5.12} and \eqref{E:5.13} in \eqref{E:5.11} and using \eqref{E:5.10}, we get

\begin{eqnarray}
&& \left(\frac{N}{p_n} - \frac{N}{2^{*}}\right) \int_{B_n} |v|^{p_n} \left( \frac{2}{1- |x|^{2}}\right)^{q_n} + \tilde \lambda \int_{B_n} 
\left( \frac{2}{1- |x|^{2}}\right)^{2} v_n^{2} = \nonumber\\
&&\frac{1}{p_n} \int_{\partial B_n}|v_n|^{p_n} \left( \frac{2}{1- |x|^{2}}\right)^{q_n}(x. \vec{n})
 + \frac{N}{2*}\int_{\partial B_n} (\grad v_n . \vec{n})v_n +   \nonumber\\
&&\int_{\partial B_n} (\grad v_n. x)(\grad v_n. \vec{n}) 
- \frac{1}{2}\int_{B_n} |\grad v_n|^{2} (x. \vec{n}) \nonumber\\
&& - \tilde \lambda \int_{B_n} \left( \frac{2}{1- |x|^{2}}\right)^{3} |x|^{2} v_n^{2}
+ \frac{\tilde \lambda}{2} 
\int_{\partial B_n} \left( \frac{2}{1- |x|^{2}}\right)^{2} v_n^{2} (x.\vec{n}) \nonumber\\
&&- \frac{q_n}{p_n} \int_{B_n}\left( \frac{2}{1- |x|^{2}}\right)^{q_n+1}
|v_n|^{p_n} |x|^{2} 
\end{eqnarray}
Ignoring the positive term on the left and the negative term on the right we get
\begin{eqnarray}\label{locpoh}
 &&\tilde \lambda \int_{B_n} \left( \frac{2}{1- |x|^{2}}\right)^{2} |v_n|^{2} \leq \frac{1}{p_n} 
\int_{\partial B_n}|v_n|^{p_n} \left( \frac{2}{1- |x|^{2}}\right)^{q_n}(x. \vec{n}) \nonumber\\
&& +\frac{N}{2*}\int_{\partial B_n} (\grad v_n . \vec{n})v_n 
+ \int_{\partial B_n} (\grad v_n. x)(\grad v_n. \vec{n}) - \frac{1}{2}\int_{\partial B_n} |\grad v_n|^{2} (x. \vec{n})\nonumber\\
&& + \frac{\tilde \lambda}{2} \int_{\partial B_n} \left(\frac{2}{1-|x|^{2}}\right)^{2} v_n^{2} (x. \vec{n})
\end{eqnarray}
Using Lemma \ref{l:5.7} we can easily show that the RHS of \eqref{locpoh} is less than or equal to $C_1 \epsilon_n^{\frac{N-2}{2}}$
for some $C_1$ independent of $n$.
Also using the decomposition \eqref{E:5.6} we can see that LHS$ \geq C_2\epsilon_n^2$. We omit the details as the proof is exactly 
the same as the one given in the proof of Lemma 6.1 in \cite{So}. Thus $\epsilon_n^2 \le C \epsilon_n^{\frac{N-2}{2}}$ where $C$ is 
independent of $n.$ This is impossible if $N\geq 7.$ This completes the proof of Theorem \ref{T:5.2} .
\section{Nonexistence}
In this section we will prove Theorem \ref{T:1.1}. The proof is based on the
Pohozaev identity. The difficulty of applying this identity 
is because of blowing up nature of the Riemmanian metric on the boundary of the
Hyperbolic ball model. So we need to have some decay estimate on the 
solution of the equation  (\ref{E:1.1}) to counter the blow up nature of the Hyperbolic metric 
on the boundary. We will use the asymptotic estimate derived in 
Section 2 to show the nonexistence of the solution.\\\\
First we will convert the equation of the Hyperbolic ball model to Euclidean Ball
model by  multiplying with the \emph{Conformal factor}. If 
$u$ solves $(\ref{E:1.1})$ , then $v = \left(\frac{2}{1 -
|x|^{2}}\right)^{\frac{N-2}{2}} u$ solves the Euclidean equation
\begin{equation}\label{E:3.1}
-\De v -\tilde \la  \left(\frac{2}{1-|x|^{2}}\right)^{2} v = |v|^{2^{*}-2} v,\quad v
\in H_{0}^{1}(\bn)
\end{equation}
where $\tilde \la =(\la - \frac{N(N-2)}{4})$.  
We will show that for $\tilde \la \leq 0$ i.e., $\la \leq \frac{N(N-2)}{4}$ has no
solution for the Eq.(\ref{E:3.1}). When $\tilde \la =0$ from the standard Pohozaev
identity we know that the equation has no solution. So it is 
enough to consider the case when $\tilde \la < 0$. Before proving the theorem we
will establish a gradient Esimate.\\

For $\var >0$ define  $A_{\var} := \{x\in \bn : 1-2\var <|x|< 1-\var \}$
\begin{proposition}\label{P:3.2}
 If $v$ satisfies Eq.(\ref{E:3.1}) with $\tilde \la \leq  0$, then 
\begin{equation}\label{E:3.2}
 \int_{A_{\var}} |\grad v|^{2} = O(\var^{\al})
\end{equation}
where $\al$ is a contant strictly greater than $1$ 
\end{proposition}
\begin{proof}  For $\var > 0$ we define a smooth function

\[
\psi_{\var}(x) = 
 \begin{cases}
 1  &\text{if $1-2\var <|x| \leq 1 -\var$}\\
 0  &\text{if $|x|\in [ 1-3\var ,1-\frac{\var}{2}]^{c}$}\\
\end{cases}
\]
such that $|\De \psi_{\var}(x)| \leq \frac{c}{\var^{2}}$.\\[5pt]
Since $v$ is a solution to the Eq.(\ref{E:3.1}), then $v$ is smooth away 
from the boundary of the Euclidean Ball and hence $\psi_{\var} v \in
C_{c}^{2}(\bn)$.Multiplying Eq. (\ref{E:3.1}) by this test
function and integrating by parts, we get
\begin{equation}\label{E:3.3}
 \int_{\bn} \grad v \grad(\psi_{\var} v) - \tilde \la \int_{\bn}
\left(\frac{2}{1-|x|^{2}}\right)^{2} \psi_{\var} v^{2}
 = \int_{\bn} |v|^{2^{*}} \psi_{\var} 
\end{equation}

By expanding we have 
\begin{eqnarray*}
&& \int_{1-3\var < |x| < 1 - \frac{\var}{2}} \langle \grad v .\grad \psi_{\var}\rangle
v + \int_{1-3\var < |x| < 1 - \frac{\var}{2}}
|\grad v|^{2} \psi_{\var}\nonumber\\
 && \leq \int_{1-3\var < |x| < 1 - \frac{\var}{2}}
|v|^{2^{*}} \psi_{\var}
 + |\tilde \la| \int_{1-3\var < |x| < 1 -
\frac{\var}{2}}\left(\frac{2}{1-|x|^{2}}\right)^{2} \psi_{\var} v^{2}
\end{eqnarray*}
Rearranging the terms we have 
\begin{eqnarray*}
 &&\int_{1-3\var < |x| < 1 - \frac{\var}{2}}|\grad v|^{2} \psi_{\var}  \leq
\int_{1-3\var < |x| < 1 - \frac{\var}{2}} |v|^{2^{*}} \psi_{\var}\\
&& + |\tilde \la| \int_{1-3\var < |x| < 1 -
\frac{\var}{2}}\left(\frac{2}{1-|x|^{2}}\right)^{2} \psi_{\var} v^{2} 
- \int_{1-3\var < |x| < 1 - \frac{\var}{2}} \grad(\frac{1}{2}v^{2})\grad \psi_{\var}\\
=  &&  \int_{1-3\var < |x| < 1 - \frac{\var}{2}} |v|^{2^{*}} \psi_{\var}
+ |\tilde \la| \int_{1-3\var < |x| < 1 -
\frac{\var}{2}}\left(\frac{2}{1-|x|^{2}}\right)^{2} \psi_{\var} v^{2} \\
&& + \int_{1-3\var < |x| < 1 - \frac{\var}{2}}(\frac{1}{2}v^{2})\De \psi_{\var}
\end{eqnarray*}
Then clearly we have 
\begin{eqnarray*}
 \int_{A_{\var}} |\grad v|^{2} && \leq \int_{1-3\var < |x| < 1 - \frac{\var}{2}}
|v|^{2^{*}}
+ |\tilde \la| \int_{1-3\var < |x| < 1 -
\frac{\var}{2}}\left(\frac{2}{1-|x|^{2}}\right)^{2} v^{2}\\
&& + \frac{c}{\var^{2}} \int_{1-3\var < |x| < 1 - \frac{\var}{2}} v^{2}
\end{eqnarray*}
Now use the estimates on $v$ from Section 2 to conclude 
\[
 \int_{A_{\var}} |\grad v|^{2} \leq O(\var^{\al})
\]
where $\al > 1$.
\end{proof}
{\bf Proof of Theorem \ref{T:1.1}.} We will prove the theorem using the Pohozaev identity. To make the test
function Smooth we introduce cut-off
 functions so that we are away from the boundary and then pass to the limit with the
help of the asymptotic estimate proved.\\
 For $\var > 0$, we define
\[
\varphi_{\var}(x) = 
 \begin{cases}
 1  &\text{if $|x| \leq 1 -2\var$}\\
 0  &\text{if $|x| \geq 1-\var$}\\
\end{cases}
\]
Assume that (\ref{E:3.1}) has a nontrivial solution $v$, then $v$ is smooth away
from the boundary of the Euclidean Ball and hence 
$(x.\grad v)\varphi_{\var} \in C_{c}^{2}(\bn)$. Multiplying Eq. (\ref{E:3.1}) by
this test function and integrate by parts, we get
\[
 \int_{\bn} \grad v.\grad ((x.\grad v)\varphi_{\var}) + |\tilde \la| \int_{\bn}
\left(\frac{2}{1-|x|^{2}}\right)^{2} v (x.\grad v)\varphi_{\var}
\]
 \begin{equation}\label{E:3.4}
  = \int_{\bn} |v|^{2^{*}-2} (x.\grad v)\varphi_{\var}
 \end{equation}
Now the RHS of (\ref{E:3.4}) can be simplified as 
\begin{eqnarray*}
 \int_{\bn} |v|^{2^{*}-2}(x.\grad v)\varphi_{\var} && = \frac{1}{2^{*}} \int_{\bn}
\langle \grad (|v|^{2^{*}}).x\rangle \varphi_{\var}\nonumber\\
&& = -\frac{N}{2^{*}} \int_{\bn} |v|^{2^{*}} \varphi_{\var} - \frac{1}{2^{*}}
\int_{\bn} |v|^{2^{*}} [x.\grad \varphi_{\var}]
\end{eqnarray*}
Using the monotone convergence theorem we get 
\begin{equation}\label{E:3.5}
\lim_{\var \rightarrow 0}  -\frac{N}{2^{*}} \int_{\bn} |v|^{2^{*}} \varphi_{\var} =
-\frac{N}{2^{*}} \int_{\bn} |v|^{2^{*}}
\end{equation}
To estimate the 2nd term of RHS we need to use the estimate on $v$  for $\tilde \la
\leq 0$ which is given by
\begin{equation}\label{E:3.6}
|v(x)| \leq C_{1}[(1 - |x|^{2})^{\frac{1 + \sqrt{1 - 4 \tilde \la}}{2}} + (1
-|x|^{2})^{2}]
 \end{equation}
Now consider 
\begin{eqnarray*}
 &&\left| \frac{1}{2^{*}} \int_{\bn} |v|^{2^{*}} [x.\grad \varphi_{\var}]\right | 
\leq \frac{c}{\var} \int_{1-2\var<|x| < 1-\var} |v|^{2^{*}}\\
&& \leq \frac{c}{\var} \int_{1-2\var< |x| < 1-\var} \left[(1 - |x|)^{(\frac{1 +
\sqrt{1-4\tilde \la}}{2})2} + (1 -|x|)^{2^{*}2} \right]
 \leq  \frac{c}{\var} \quad \var^{\al}
\end{eqnarray*}
where $\al > 1$. Then letting $\var \rightarrow 0$  in the above we have
\begin{equation}\label{E:3.7}
 \left| \frac{1}{2^{*}} \int_{\bn} |v|^{2^{*}} [x.\grad \varphi_{\var}]\right |
\rightarrow 0
\end{equation}
Hence we have from (\ref{E:3.5}) and (\ref{E:3.7}),  
\begin{equation}\label{E:3.8}
 \lim_{\var \rightarrow 0} [RHS] = -\frac{N}{2^{*}} \int_{\bn} |v|^{2^{*}}
\end{equation}
By direct calculation and integration by parts, LHS of (\ref{E:3.4}) simplifies as 
\begin{eqnarray*}
&& \mbox{LHS}  = \int_{\bn} |\grad v|^{2} \varphi_{\var} + \sum_{i =1}^{N} 
\sum_{j=1}^{N} \frac{1}{2} \int_{\bn} (v_{x_{i}})^{2}_{x_{j}}
\varphi_{\var} x_{j} + \int_{\bn} \langle x.\grad v \rangle \langle \grad v .\grad
\varphi_{\var} \rangle\\
&& + 2 |\tilde \la| \int_{\bn} \left(\frac{2}{1-|x|^{2}}\right)^{2} v^{2}
\varphi_{\var} -  |\tilde \la| \int_{\bn} 
\left(\frac{2}{1-|x|^{2}}\right)^{3} v^{2} \varphi_{\var}  \\
&&- \frac{|\tilde \la|
N}{2} \int_{\bn} 
\left(\frac{2}{1-|x|^{2}}\right)^{2} v^{2} \varphi_{\var}
 - \frac{|\tilde \la|}{2} \int_{\bn} \langle x.\grad \varphi_{\var} \rangle 
\left(\frac{2}{1-|x|^{2}}\right)^{2} v^{2} \\
&& = - \frac{N-2}{2} \int_{\bn} |\grad v|^{2} \varphi_{\var} - \frac{1}{2}
\int_{\bn} |\grad v|^{2} \langle x. \grad \varphi_{\var} \rangle
+ \int_{\bn} \langle x.\grad v \rangle \langle \grad v .\grad \varphi_{\var} \rangle\\
&& + 2 |\tilde \la| \int_{\bn} \left(\frac{2}{1-|x|^{2}}\right)^{2} v^{2}
\varphi_{\var} -  |\tilde \la| \int_{\bn} 
\left(\frac{2}{1-|x|^{2}}\right)^{3} v^{2} \varphi_{\var}  \\
&&- \frac{|\tilde \la|
N}{2} \int_{\bn} 
\left(\frac{2}{1-|x|^{2}}\right)^{2} v^{2} \varphi_{\var}
 - \frac{|\tilde \la|}{2} \int_{\bn} \langle x.\grad \varphi_{\var} \rangle
\left(\frac{2}{1-|x|^{2}}\right)^{2} v^{2} 
\end{eqnarray*}
Using the monotone convergence theorem we get  
\begin{equation}\label{E:3.9}
  \lim_{\var \rightarrow 0 }- \frac{N-2}{2} \int_{\bn} |\grad v|^{2} \varphi_{\var}
= - \frac{N-2}{2} \int_{\bn} |\grad v|^{2}
\end{equation}
Using the monotone convergence theorem we get
\begin{equation}\label{E:3.10}
 \lim_{\var \rightarrow 0} 2 |\tilde \la| \int_{\bn}
\left(\frac{2}{1-|x|^{2}}\right)^{2} v^{2} \varphi_{\var} = 
2 |\tilde \la| \int_{\bn} \left(\frac{2}{1-|x|^{2}}\right)^{2} v^{2}
\end{equation}
Again using the monotone convergence theorem
\begin{equation}\label{E:3.11}
 \lim_{\var \rightarrow 0} - |\tilde \la| \int_{\bn}
\left(\frac{2}{1-|x|^{2}}\right)^{3} v^{2} \varphi_{\var}=
- |\tilde \la| \int_{\bn} \left(\frac{2}{1-|x|^{2}}\right)^{3} v^{2} 
\end{equation}
Similarly again using the monotone convergence theorem 
\begin{equation}\label{E:3.12}
 \lim_{\var \rightarrow 0} - \frac{|\tilde \la| N}{2}
\int_{\bn}\left(\frac{2}{1-|x|^{2}}\right)^{2} v^{2} \varphi_{\var} =
- \frac{|\tilde \la| N}{2} \int_{\bn}\left(\frac{2}{1-|x|^{2}}\right)^{2} v^{2}
\end{equation}

Now consider the term
\begin{eqnarray*}
 &&\left|- \frac{|\tilde \la|}{2} \int_{\bn} \langle x.\grad \varphi_{\var} \rangle
\left(\frac{2}{1-|x|^{2}}\right)^{2} v^{2} \right| \leq 
\frac{c}{\var} \int_{1-2\var < |x| < 1- \var} \left(\frac{2}{1- |x|^{2}}\right)^{2}
v^{2}\\
&& \leq \frac{c}{\var} \int_{1-2\var < |x| < 1- \var} \left(\frac{2}{1-
|x|}\right)^{2} [(1-|x|)^{1 + \sqrt{1 - 4\tilde \la}} +
(1 - |x|)^{4}]
 \leq \frac{c}{\var} \quad \var^{\al}
\end{eqnarray*}
where $\al > 1$. Then letting $\var \rightarrow 0$ in the above we get 
\begin{equation}\label{E:3.13}
 \left|- \frac{|\tilde \la|}{2} \int_{\bn} \langle x.\grad \varphi_{\var} \rangle
\left(\frac{2}{1-|x|^{2}}\right)^{2} v^{2} \right|
\rightarrow 0
\end{equation}
Now consider the remaining  term
\begin{eqnarray*}
 |\int_{\bn} \langle x.\grad v \rangle \langle \grad v .\grad \varphi_{\var}
\rangle| && \leq \int_{\bn} |\grad v|^{2} |\grad \varphi_{\var}|\\
&& \leq \frac{c}{\var} \int_{1-2\var < |x| < 1-\var} |\grad v|^{2}
\end{eqnarray*}
Hence by the gradient estimate the above term goes to zero as $\var \rightarrow 0$.\\
Using  (\ref{E:3.9}), (\ref{E:3.10}), (\ref{E:3.11}),  (\ref{E:3.12}) and
(\ref{E:3.13}) we have 
\begin{eqnarray}\label{E:3.14}
 &&\lim_{\var \rightarrow 0}\mbox{[LHS]} =  - \frac{N-2}{2} \int_{\bn} |\grad v|^{2}
+ 2 |\tilde \la| \int_{\bn}
 \left(\frac{2}{1-|x|^{2}}\right)^{2} v^{2}\nonumber\\
&& - |\tilde \la| \int_{\bn} \left(\frac{2}{1-|x|^{2}}\right)^{3} v^{2} -
\frac{|\tilde \la| N}{2} \int_{\bn}\left(\frac{2}
{1-|x|^{2}}\right)^{2} v^{2}
\end{eqnarray}
Substituting (\ref{E:3.8}) and (\ref{E:3.14}) in (\ref{E:3.4}), and using
Eq.(\ref{E:3.1}), we get
\begin{equation}\label{E:3.15}
 - 4 \tilde \la \int_{\bn} \frac{(1 + |x|^{2})}{(1- |x|^{2})^{3}} v^{2} = 0
\end{equation}
which implies $v = 0$.

\section{Existence}
In this section we will prove Theorem \ref{T:1.3}. In view of Corollory \ref{comp-sol}, it is enough to construct 
infinitely many solutions for subcritical problems which are bounded in $H^1_r(\bn)$. Infinitely many solutions for the
subcritical problem have been established in \cite{BS-2}, however we do not have any idea about their boundedness.
In this section we will prove the existence of sign changing solutions for the subcritical problem with an estimate on
 the Morse index from below by applying the abstract theorem of Schechter and Zou \cite{Sz}.\\\\
We fix a $p_{0} \in (2, 2^{*})$ and  choose a sequence $p_{n}$ in $(p_{0},2)$ such
that $p_{n} \rightarrow 2^{*}$. 
Consider the problem
\begin{equation}\label{E:4.1}
 -\De_{\bn} u = \la u + |u|^{p_{n}-2}u ,\quad u \in H_{r}^{1}(\bn)
\end{equation}
then we have :
\begin{theorem}\label{T:4.1}  Fix $\la \in (\frac{N(N-2)}{4}, (\frac{N-1}{2})^{2})$, then for every
$n$ the Equation (\ref{E:4.1}) has infinitely many radial sign changing solutions $\{u_k^n\}_{k=1}^\infty$
such that for each $k$, the sequence $\{u_k^n\}_{n=1}^\infty$ is bounded in $H_{r}^{1}(\bn)$ and the augmented Morse index
of $u_k^n$ on the space $H_{r}^{1}(\bn)$ is greater then or equal to $k$.
\end{theorem}
To prove the theorem we have to show that the functional
\[
 J_{n,\la}(u) = \frac{1}{2} \int_{\bn} |\grad_{\bn} u|^{2} dV_{\bn}  - \frac{\la}{2}
\int_{\bn} u^{2}dV_{\bn} - \frac{1}{p_{n}} \int_{\bn}
|u|^{p_{n}}dV_{\bn} 
\]
defined on $H^1(\bn)$ has infinitely many critical points $\{u_k^n\}_{n=1}^\infty$. Because of the principle of
symmetric criticality (\cite{P}), enough to find the critical points of $J_{n,\la}$ on $H_{r}^{1}(\bn)$. The
augmented Morse index of $u_k^n$ on the space $H_{r}^{1}(\bn)$ is the dimention of the largest subspace of $H_{r}^{1}(\bn)$
where $J_{n,\la}^{\prime\prime}(u_k^n)$ is nonpositive definite.
 We will prove the theorem by working with its conformal version \eqref{conc-eu}. We will show that
the functional, 
\[
 G_{n,\la}(v) = \frac{1}{2} \int_{\bn} |\grad v|^{2}  - \frac{\tilde \la}{2}
\int_{\bn} \left(\frac{2}{1 -|x|^{2}}\right)^{2} v^{2} 
- \frac{1}{p_{n}} \int_{\bn}
|v|^{p_{n}}\left(\frac{2}{1-|x|^{2}}\right)^{q_{n}} 
\]
defined on $H^1_{0,r}(\bn)$ satisfies all the assumptions of Theorem 2 of \cite{Sz} where $\tilde \la$ is 
as in \eqref{conc-eu}  .\\\\
 Let $0 < \la_{1} < \la_{2}
\leq \la_{3}\ldots \leq \la_{k}\leq \ldots$ be the eigen values of $-\De$ on $H^1_{0,r}(\bn)$ and
$\varphi_{k}(x)$ be the eigen functions corresponding
to $\la_{k}$. Denote $E_{k} := \mbox{span}\{\varphi_{1},\varphi_{2},\ldots ,
\varphi_{k}\}$. Then $H^1_{0,r}(\bn)= \overline{\cup_{k=1}^{\I} E_{k}}$, dim$E_{k} =k$ and $E_{k} \subset E_{k+1}$.\\ 
For each $p_{n}\in (2, 2^{*})$, we define 
$$||u||_{*} = \left[\int_{\bn}|u|^{p_{n}}\left(\frac{2}{1-|x|^{2}}\right)^{q_{n}}\right]^{\frac{1}{p_{n}}}, u \in H^1_{0,r}(\bn)$$ 
then from \eqref{PSE-C} we get $||v||_{*} \leq C ||v||$ for all $v \in H_{0,r}^{1}(\bn)$ for some constant $C>0$. Moreover using the radial estimate the 
embedding of $H^1_{0,r}(\bn)$ in to $(H^1_{0,r}(\bn),||.||_{*}) $ is compact.\\
Now define,
$$P:= \{v \in H_{0,r}^{1}(\bn) : v\geq 0\}$$ Also for $\mu > 0$, define 
$$D(\mu) := \{v\in H^1_{0,r}(\bn): \mbox{dist}(v,P) < \mu\} ,\;D^{*}
:= D(\mu) \cup (-D(\mu)).$$ Also denote the set of all critical points by
$$K_{n}^{\la}:= \{v\in H_{0,r}^{1}(\bn): G'_{n,\la}(v) = 0\}$$
Clealy $G_{n,\la} \in C^{2}((H_{0}^{1}, ||.||), \mathbb{R})$
 is an even functional which maps bounded sets to bounded sets in terms of the norm $||.||$. The gradient $G'_{n,\la}$ is of the form 
$G'_{n,\la}(v) = v - K_{n,\la}(v)$, where $K_{n,\la} : E \rightarrow E$ is a
continuous operator. Moreover
\begin{proposition}\label{P:4.5}
 For any $\mu_{0} > 0$ small enough, we have that $K_{n,\la}(D(\mu_{0}))\subset
D(\mu) \subset D(\mu_{0})$ for some $\mu \in (0,\mu_{0})$ for 
each $n,\la$  with $ \frac{N(N-2)}{4}<\la <\frac{(N-1)^{2}}{4}$. Moreover,
$D(\mu_{0}) \cap K_{n}^{\la} \subset P$
\end{proposition}
\begin{proof} First note that $K_{n,\la}(v)$ can be decomposed as $K_{n,\la}(v) = L(v) + W(v)$
where $L(v), W(v) \in E$ are the unique solutions of the equations
\[
 -\De (L(v))= \tilde \la v \left(\frac{2}{1-|x|^{2}}\right)^{2} \mbox{and} \quad
-\De(W (v)) = |v|^{p_{n}-2} v \left(\frac{2}{1-|x|^{2}}\right)^
{q_n}.
\]
In other words, $L(v)$ and $W(v)$ are uniquely determine by the relations
\begin{equation}\label{E:4.3}
 \langle Lv, u\rangle = \tilde \la \int_{\bn} u v
\left(\frac{2}{1-|x|^{2}}\right)^{2}\;,\;  \langle W(v),u \rangle = 
\int_{\bn} |v|^{p_{n}-2} v u \left(\frac{2}{1-|x|^{2}}\right)^{q_n}
\end{equation}
By Maximum Principle, $L(v)\in P$ and $W(v)\in P$ if $v\in P$.\\
Now we will estimate $||L(v)||$. We have
\begin{eqnarray*}
\langle Lv, Lv \rangle && = \tilde \la \int_{\bn} v Lv
\left(\frac{2}{1-|x|^{2}}\right)^{2} dx\\
&& \leq \tilde \la \left(\int_{\bn} \left(\frac{2}{1-|x|^{2}}\right)^{2}
v^{2}\right)^{\frac{1}{2}}  
\left(\int_{\bn} \left(\frac{2}{1-|x|^{2}}\right)^{2} |Lv|^{2}\right)^{\frac{1}{2}}\\
&& \leq 4 \tilde \la ||v|| ||Lv||
\end{eqnarray*}
thanks to Lemma (\ref{P:4.4}). Thus
$||Lv|| \le 4 \tilde \la ||v|| $
where $4\tilde \la < 1$. Let $v\in H_{0,r}^{1}(\bn)$ and $u\in P$ be such that
dist$(v,P) = ||u-v||$, then 
\begin{equation}\label{E:4.4}
 \mbox{dist}(Lv,P) \leq ||Lv-Lu|| \leq 4 \tilde \la||u-v|| \leq 4 \tilde \la \mbox{dist}(v,P)
\end{equation}
where $4 \tilde \la < 1$.\\
To estimate the distance between $W(v)$ and $P$,set $v^{-}:= \mbox{min}\{ v,0\}$. Then
\begin{eqnarray*}
&&\mbox{dist} (W(v),P) ||W(v)^{-}||  \leq ||W(v)^{-}||^{2} \leq \langle W(v), W(v)^{-} \rangle\\
&& = \int_{\bn} |v|^{p_{n}-2} v W(v)^{-}
\left(\frac{2}{1-|x|^{2}}\right)^{q_n} 
 \leq \int_{\bn} |v^{-}|^{p_{n}-1} |W(v)^{-}| 
\left(\frac{2}{1-|x|^{2}}\right)^{q_n}\\
&& \leq \left( \int_{\bn} |v^{-}|^{p_{n}} 
\left(\frac{2}{1-|x|^{2}}\right)^{q_n}\right)^{\frac{p_n-1}{p_n}}
 \left( \int_{\bn} |W(v)^{-}|^{p_{n}} 
\left(\frac{2}{1-|x|^{2}}\right)^{q_n}\right)^{\frac{1}{p_n}}\\
&&\leq C \left( \int_{\bn} |v^{-}|^{p_{n}} 
\left(\frac{2}{1-|x|^{2}}\right)^{q_n}\right)^{\frac{p_n-1}{p_n}}||W(v)^{-}||
\end{eqnarray*}
Now using 
$$ \int_{\bn} |v^{-}|^{p_{n}} 
\left(\frac{2}{1-|x|^{2}}\right)^{q_n} = \min\limits_{u\in P}\int_{\bn} |v-u|^{p_{n}} 
\left(\frac{2}{1-|x|^{2}}\right)^{q_n}\leq C\min\limits_{u\in P}||v-u||$$
we get
\[
 \mbox{dist} (W(v),P) \leq C[\mbox{dist}(v,P)]^{p_{n}-1} 
\quad \forall v\in H_{0}^{1}(\bn)
\]
Choose $4 \tilde \la < \nu < 1$. Then there exists $\mu_{0}$ such that, if $\mu \leq
\mu_{0}$,
\begin{equation}\label{E:4.5}
 \mbox{dist}(W(v),P) \leq (\nu - 4 \tilde \la) \mbox{dist}(v,P) \quad \mbox{for
all}\quad  v\in D(\mu).
\end{equation}
 Fix $\mu \leq \mu_{0}$. Inequalities (\ref{E:4.4}) and (\ref{E:4.5}) yeild
\begin{eqnarray*}
 \mbox{dist}(K_{n,\la}(v),P) && \leq \mbox{dist}(L(v),P) + \mbox{dist} (W(v),P)\\
&& \leq \nu \mbox{dist} (v,P)
\end{eqnarray*}
for all $v\in D(\mu)$. This proves of the Proposition.
\end{proof}

\begin{proposition}\label{P:4.7}
 For each k, $\lim_{||v||\rightarrow \I, v\in E_{k}} G_{n,\la}(v) = -\I$
\end{proposition}
Proof: Since $E_k$ is finite dimensional, there exists a constant $C>0$ such that $||v|| \leq C ||v||_{*}$ 
for all $ v \in E_k.$ Thus
\[
 G_{n,\la}(v)  \leq \frac{1}{2} ||v||^{2} - C ||v||^{p_{n}} ,\;\forall \;  v \in E_k
\]
Since $p_{n} > 2$, we have $\lim_{||v||\rightarrow \I, v\in E_{k}} G_{n,\la}(v) = -\I$.
\begin{proposition}\label{P:4.8}
 For any $\al_{1}, \al_{2} > 0$, there exist an $\al_{3}$ depending on $\al_{1}$ and
$\al_{2}$ such that $||v||\leq \al_{3}$ for all
$v\in G_{n,\la}^{\al_{1}} \cap \{ v\in H_{0}^{1}(\bn): ||v||_{*} \leq \al_{2}\}$  where
$G_{n,\la}^{\al_{1}} = \{ v\in H_{0,r}^{1}(\bn): G_{n,\la} \leq \al_{1}\}$ 
\end{proposition}
\begin{proof} The proposition follows since
 $$ \frac{1-4\tilde \la}{2}||v||^2\le G_{n,\la}(v) + \frac{1}{p_n}||v||_{*} ^{p_n}$$
\end{proof}

{\bf Proof of Theorem \ref{T:4.1}} The above discussions and Propositions \ref{P:4.5}, \ref{P:4.7} \ref{P:4.8} tells us 
that $G_{n,\la} $ satisfies all the conditions of Theorem 2 in \cite{Sz}. Thus $G_{n,\la} $ has a sign changing critical point
 $v_k^n\in H_{0,r}^{1}(\bn)$ at a level $C(n,\la,k)$ and $C(n,\la,k) \leq \mbox{sup}_{E_{k+1}} G_{n,\la}$ and the augmented Morse
index $m^{*}(v_k^n)$ of $v_k^n$ is $\geq k$. We claim that:\\
{\it Claim :} There exists a constant $T_{1} > 0$ independent of $k$ and $n$ such that 
\[
 \mbox{sup}_{E_{k+1}} G_{n,\la} \leq T_{1} \la_{k+1}^{\frac{p_{0}}{2(p_{0} -2)}}
\]
{\it Proof of claim :}  The defination of $E_{k+1}$
implies that $||v||^{2}\leq  \la_{k+1} ||v||^{2}_{2}$. Note that 
with $p_{n} > p_{0}$, we have $||v||_{p_{0}} \leq D_{1} ||v||_{p_{n}}$, where $D_{1}
> 0$ is a constant independent of $n$ and $k$. Therefore,
\begin{eqnarray*}
 G_{n,\la}(v) && \leq \frac{1}{2} \int_{\bn} |\grad v|^{2}  - \frac{\tilde \la}{2}
\int_{\bn} \left(\frac{2}{1 -|x|^{2}}\right)^{2} v^{2} 
- \frac{1}{p_{n}} \int_{\bn} |v|^{p_{n}}\\
&& \leq  \frac{1}{2} \int_{\bn} |\grad v|^{2}   - \frac{1}{p_{n}} \int_{\bn}
|v|^{p_{n}}\\
&& \leq  \frac{1}{2} \int_{\bn} |\grad v|^{2} -D_{2} \int_{\bn} |v|^{p_{0}} + D_{3}
\end{eqnarray*}
where $D_{2}>0, D_{3} > 0$ are constant, independent of $n$ and $k$. Since there
exist a constant $D_{4} > 0$ such that $||v||_{2} \leq D_{4} ||v||
_{p_{0}}$, therefore we may have $D_{5} > 0$ such that $||v||^{p_{0}} \leq D_{5}
\la_{k+1}^{p_{0}/2} ||v||_{p_{0}}^{p_{0}}$ for all 
$v \in E_{k+1}$. Then 
\begin{eqnarray*}
 G_{n,\la}(v) && \leq \frac{1}{2} ||v||^{2} - D_{6}\la_{k+1}^{-p_{0}/2}
||v||^{p_{0}} + D_{3}\\
&& \leq D_{7}\la_{k+1}^{\frac{p_{0}}{2(p_{0} -2)}} + D_{3}\\
&& \leq T_{1}\la_{k+1}^{\frac{p_{0}}{2(p_{0} -2)}}
\end{eqnarray*}
where $D_{i}(i=1, \ldots ,7)$ and $T_{1}$ are positive constants independent of $k$
and $n$.\\\\
Also note that energy of any critical point of $G_{n,\la}$ is positive. Thus 
$G_{n,\la}(v_k^n) \in [0,T_{1} \la_{k+1}^{\frac{p_{0}}{2(p_{0} -2)}} ].$
This immediately implies that the sequence $\{v_k^n\}_{n=1}^\infty$ is bounded in $H_0^{1}(\bn)$ for each $k.$ Now
$u_k^n= (\frac{1-|x|^2}{2})^{\frac{N-2}{2}}v_k^n$ satisfies all the conclusions
of Theorem \ref{T:4.1}. Moreover $J_{n,\la}(u_k^n) = G_{n,\la}(v_k^n)$. \\

{\bf Proof of Theorem \ref{T:1.3}}. Using Corollory \ref{comp-sol} and Theorem \ref{T:4.1}, we get
a sequence $\{u_k\}_{k=1}^\infty$ of solutions of our original problem with energy  
$C(\la,k) \in [0,T_{1} \la_{k+1}^{\frac{p_{0}}{2(p_{0} -2)}}]$. It remains 
to show that infinitely many 
$u_k$'s are different. This folllows if we show that the energy of $u_k$ goes to infinity as $k\rightarrow \infty.$\\
Suppose not,
then $\lim_{k\rightarrow \I}C(\la,k) = c'< \I$.
For any $k\in \mathbb{N}$ we may find an $n_{k}$(assume $n_{k} > k$) such that
$|C(n_{k,\la,k}) - C(\la,k)| < \frac{1}{k}$. It follows
 that $\lim_{k\rightarrow \I} C(n_{k},\la, k) = \lim_{k\rightarrow \I}
C(\la,k) = c'<\I$. Hence, 
 $\{u^{n_{k}}_k\}_{k\in \mathbb{N}}$ is bounded in $H^1_r(\bn)$ and hence satisfies the uniform bound
given by Theorem \ref{T:5.2}. Therefore the augmented Morse index of $u^{n_k}_k$ remains bounded which contradicts the fact that 
the augmented Morse index of $u^{n_k}_k$ is greater than or equal to $k$.
 Thus $\lim_{k\rightarrow \I}C(\la,k) = \I$ and hence infinitely many $u_k$'s are different. Moreover they are sign changing as 
the radial positive solutions are unique (see \cite{MS}, Theorem 1.3). This completes the proof.  
\section{Appendix}

Let $\bn:=\{x\in\rn: |x|<1\}$ denotes the unit disc in $\rn$. The space $\bn$ endowed 
with the Riemannian metric $g$ given by $g_{ij}=(\frac{2}{1-|x|^2})^2\de_{ij}$ is called
the ball model of the Hyperbolic space. For more details on hyperbolic geometry we refer to \cite{JR}.\\\\ 
We will denote the associated hyperbolic 
volume by $dV_{\bn}$ and is given by  $dV_{\bn}=(\frac{2}{1-|x|^2})^N dx$.
The hyperbolic gradient $\na_{\bn}$ and the hyperbolic Laplacian $\De_{\bn}$ are given by
$$
 \na_{\bn}=(\frac{1-|x|^2}{2})^2\na,\ \ \  \De_{\bn}=(\frac{1-|x|^2}{2})^2\De+(N-2)\frac{1-|x|^2}{2}<x,\na>
$$
Let $H^1(\bn)$ denotes the Sobolev space on $\bn$ with the above metric $g$, then we have
$H^1(\bn) \hookrightarrow L^p(\bn)$ for $2\le p\le \frac{2N}{N-2}$ when $N\ge 3$ and
$p\ge 2$ when $N=2$. In fact we have the following Poincar\'{e}-Sobolev inequality (see \cite{MS}):\\\\
For every $N \geq 3$ and every $p\in (2, \frac{2N}{N-2}] $
there is an optimal constant \\
$S_{N,p,\la }>0$ such that
\begin{equation}\label{PSE}
 S_{N,p,\la} \left(\int\limits_{\bn}  |u|^{p} dV_{\bn} \right)^{\frac{2}{p}} \leq 
\int\limits_{\bn} \left[ |\nabla_{\bn}  u|^2  - {\frac{(n-1)^2 }{4}} u^2 \right]  dV_{\bn} \quad 
\end{equation}
for every $ u\in H^1(\bn).$\\\\
As an immediate consequence we get :\\
\begin{lemma}\label{PSE-l} For any $\la < {\frac{(n-1)^2 }{4}}$, $||u||_{\la}$ defined by
$$||u||_{\la}^2 \;:=\; \int\limits_{\bn} \left[ |\nabla_{\bn}  u|^2  - \la u^2 \right]dV_{\bn} $$ 
is an equivalent norm in $H^1(\bn)$. 
\end{lemma}
Making a conformal change of the metric, we can get a Euclidean version of \eqref{PSE} on the Euclidean
Sobolev space $H^1_0{(\bn)}.$
\begin{lemma}\label{P:4.4} There is an optimal constant 
$S_{N,p,\la }>0$ such that
\begin{equation}\label{PSE-C}
 S_{N,p,\la} \left(\int\limits_{\bn}  |v|^{p} (\frac{2}{1-|x|^2})^q dx\right)^{\frac{2}{p}} \leq 
\int\limits_{\bn} \left[ |\nabla  v|^2  - {\frac{1}{4}}  (\frac{2}{1-|x|^2})^2 v^2\right]  dx \quad 
\end{equation}
for every $ v\in H^1_0(\bn)$ where $p\in (2, \frac{2N}{N-2}] $ and $q= \frac{2N-p(N-2)}{2}.$
\end{lemma}
\begin{proof} Put $u = \left(\frac{2}{1 - |x|^{2}}\right)^{-\frac{N-2}{2}} v$ in \eqref{PSE} will
establish the lemma.
\end{proof}

\end{document}